\documentclass[a4paper,11pt]{article}
\usepackage{amsmath, amsfonts, amscd, amssymb, amsthm, enumerate}

\def\bfB{\mathbf{B}}

\newcommand{\PA}{(\mathrm{A})}
\newcommand{\PAplus}{(\mathrm{A^+})}
\newcommand{\PM}{(\mathrm{M})}

\newcommand{\Mat}{\operatorname{M}}

\newcommand{\id}{\operatorname{id}}

\newcommand{\Ker}{\operatorname{Ker}}
\newcommand{\End}{\operatorname{End}}
\newcommand{\Vect}{\operatorname{span}}
\newcommand{\im}{\operatorname{Im}}

\newcommand{\tr}{\operatorname{tr}}

\newcommand{\rk}{\operatorname{rk}}

\renewcommand{\setminus}{\smallsetminus}

\def\F{\mathbb{F}}

\def\C{\mathbb{C}}

\def\N{\mathbb{N}}

\def\lcro{\mathopen{[\![}}
\def\rcro{\mathclose{]\!]}}

\theoremstyle{definition}
\newtheorem{Def}{Definition}

\theoremstyle{plain}
\newtheorem{theo}{Theorem}
\newtheorem{prop}[theo]{Proposition}
\newtheorem{cor}[theo]{Corollary}
\newtheorem{lemma}[theo]{Lemma}

\theoremstyle{plain}

\theoremstyle{remark}
\newtheorem{Rems}{Remarks}
\newtheorem{Rem}[Rems]{Remark}
\newtheorem{ex}[Rems]{Example}

\title{Sums of three quadratic endomorphisms of an infinite-dimensional vector space}
\author{Cl\'ement de Seguins Pazzis\footnote{Universit\'e de Versailles Saint-Quentin-en-Yvelines, Laboratoire de Math\'ematiques
de Versailles, 45 avenue des Etats-Unis, 78035 Versailles cedex, France}
\footnote{e-mail address: dsp.prof@gmail.com}}

\begin{document}

\thispagestyle{plain}

\maketitle

\begin{abstract}
Let $V$ be an infinite-dimensional vector space over a field.
In a previous article \cite{dSPSum4}, we have shown that every endomorphism of $V$ splits into the sum of four square-zero
ones but also into the sum of four idempotent ones. Here, we study decompositions into sums of three endomorphisms
with prescribed split annihilating polynomials with degree $2$.
Except for endomorphisms that are the sum of a scalar multiple of the identity and of a finite-rank endomorphism,
we achieve a simple characterization of such sums. In particular, we give a simple characterization
of the endomorphisms that split into the sum of three square-zero ones, and we prove that every endomorphism of $V$
is a linear combination of three idempotents.
\end{abstract}



\section{Introduction}

Throughout the article, $\F$ denotes an arbitrary field and $V$ is an infinite-dimensional vector space over $\F$, whose
algebra of endomorphisms we denote by $\End(V)$.
An endomorphism $u$ of $V$ is called \textbf{quadratic} whenever there exists a polynomial $p(t) \in \F[t]$
with degree $2$ such that $p(u)=0$. Special cases of quadratic endomorphisms are the square-zero ones,
the idempotent ones, and the involutions.
Let $p_1,\dots,p_n$ be split polynomials with degree $2$ over $\F$.
We call an endomorphism $u$ of $V$ a $(p_1,\dots,p_n)$\textbf{-sum} whenever there exists an $n$-tuple $(u_1,\dots,u_n)$
of endomorphisms of $V$ such that
$$u=\sum_{k=1}^n u_k \quad \text{and} \quad \forall k \in \lcro 1,n\rcro, \; p_k(u_k)=0.$$
We adopt a similar definition for square matrices over $\F$.

Likewise, a scalar $\lambda$ is called a $(p_1,\dots,p_n)$\textbf{-sum}
whenever there exists an $n$-tuple $(x_1,\dots,x_n)\in \F^n$ such that
$$\lambda=\sum_{k=1}^n x_k \quad \text{and} \quad \forall k \in \lcro 1,n\rcro, \; p_k(x_k)=0.$$

In a recent work \cite{dSPSum4}, we have obtained the following general result:

\begin{theo}\label{theo4}
Let $(p_1,p_2,p_3,p_4)$ be a four-tuple of split polynomials with degree $2$ over $\F$.
Then, every endomorphism of an infinite-dimensional vector space over $\F$ is a $(p_1,p_2,p_3,p_4)$-sum.
\end{theo}

In particular, every endomorphism of an infinite-dimensional vector space
is the sum of four square-zero endomorphisms, but also of four idempotents, of two idempotents and two square-zero endomorphisms,
etc. This contrasts with two results that were previously known:
\begin{itemize}
\item If $V$ is a finite-dimensional vector space, then an endomorphism of $V$ is the sum of four square-zero endomorphisms if and only if its trace equals zero (see \cite{WangWu} for the case of a complex vector space, and \cite{dSP3squarezero} for the general case).
\item If $\F=\C$ and $V$ is a Hilbert space, then any bounded operator on $V$ is the sum of five square-zero bounded operators \cite{PearcyTopping}, and a bounded operator on $V$ is the sum of four square-zero bounded operators if and only if it is a commutator
    \cite{WangWu}.
\end{itemize}
Compared to the latter result, Theorem \ref{theo4} is purely algebraic, and no structure from analysis is involved.

In \cite{dSPSum4}, it was shown through various examples that four summands are necessary in
Theorem \ref{theo4}. To be more precise, if we have three split polynomials $p_1,p_2,p_3$ with degree $2$ over $\F$,
in general there exist endomorphisms of $V$ that fail to be $(p_1,p_2,p_3)$-sums.
Thus, a natural question is whether a simple characterization of $(p_1,p_2,p_3)$-sums
can be obtained. In this work, we shall obtain an answer that is very close to a positive one.
More precisely, we shall obtain such a characterization if we exclude very specific endomorphisms,
specifically those that split into $\lambda \id_V+w$ where $\lambda$ is a scalar and $w$ is a finite-rank endomorphism of $V$.
For such special endomorphisms, no characterization appears possible in general, but in the special case
when $p_1=p_2=p_3=t^2$ we shall nevertheless succeed in obtaining one, leading to a complete characterization of the
sums of three square-zero endomorphisms. In addition, we will give a full characterization of the endomorphisms that
split into the sum of three idempotents if the underlying field has characteristic $2$,
and we will prove that every endomorphism of an infinite-dimensional vector space is a linear combination of three idempotents,
a result that was known to hold over finite-dimensional vector spaces \cite{Rabanovich,dSPLC3}.

\section{Main results, and the strategy}

\subsection{Main results}

\begin{Def}
Let $V$ be an infinite-dimensional vector space and $u \in \End(V)$.
A scalar $\lambda$ is called a \textbf{dominant eigenvalue} of $u$ if $\rk(u-\lambda\,\id_V)<\dim V$.
\end{Def}

Note that this implies that $\lambda$ is actually an eigenvalue of $u$, that is $\Ker(u-\lambda \,\id_V) \neq \{0\}$.
Moreover, $u$ has at most one dominant eigenvalue.
Indeed, given distinct scalars $\lambda$ and $\mu$
we have $\Ker(u-\mu \id_V) \subset \im(u-\lambda \id_V)$, and hence
$\rk(u-\mu \id_V)+\rk(u-\lambda \id_V) \geq \dim V$: since $V$ is infinite-dimensional, it follows that at most one of
$\rk(u-\mu \id_V)$ and $\rk(u-\lambda \id_V)$ is less than $\dim V$.

\vskip 3mm
Here is our first major result:

\begin{theo}\label{theo3}
Let $u$ be an endomorphism of an infinite-dimensional vector space over $\F$, with no dominant eigenvalue.
Let $p_1,p_2,p_3$ be split polynomials of degree $2$ over $\F$.
Then, $u$ is a $(p_1,p_2,p_3)$-sum.
\end{theo}

Knowing this result, it only remains to understand when an endomorphism with a dominant eigenvalue is a $(p_1,p_2,p_3)$-sum.
To do this, we recall that the \textbf{trace} of a monic polynomial $p \in \F[t]$
with degree $n>0$ is defined as the opposite of the coefficient of $p$ on $t^{n-1}$.
The trace of a nonconstant polynomial $p$ with leading coefficient $\alpha$ is defined as the one of $\alpha^{-1} p$,
and denoted by $\tr p$.

When we have an endomorphism $u$ with a dominant eigenvalue $\lambda$, the
following result gives a necessary condition on $\lambda$ for $u$ to be a $(p_1,p_2,p_3)$-sum.

\begin{theo}\label{dominanteigenvalueCN}
Let $V$ be an infinite-dimensional vector space over $\F$.
Let $p_1,p_2,p_3$ be split polynomials with degree $2$ over $\F$.
Let $u$ be an endomorphism of $V$ with a dominant eigenvalue $\lambda$, and
assume that $u$ is a $(p_1,p_2,p_3)$-sum.
Then, $\lambda$ is a $(p_1,p_2,p_3)$-sum or $2\lambda=\tr p_1+\tr p_2+\tr p_3$.
\end{theo}

The above condition turns out to be sufficient unless $u-\lambda\,\id_V$ has finite rank:

\begin{theo}\label{dominanteigenvalueCS}
Let $V$ be an infinite-dimensional vector space over $\F$.
Let $p_1,p_2,p_3$ be split polynomials with degree $2$ over $\F$.
Let $u$ be an endomorphism of $V$ with a dominant eigenvalue $\lambda$
such that $u-\lambda\,\id_V$ has infinite rank.
Assume that $\lambda$ is a $(p_1,p_2,p_3)$-sum
or $2\lambda=\tr p_1+\tr p_2+\tr p_3$.
Then, $u$ is a $(p_1,p_2,p_3)$-sum.
\end{theo}

Hence, it only remains to understand when the sum of $\lambda \id_V$
with a finite-rank endomorphism is a $(p_1,p_2,p_3)$-sum. We will show in Section \ref{finiteranksection} that
this amounts to determine, given a scalar $\lambda$ that satisfies the condition from Theorem \ref{dominanteigenvalueCS},
for which square matrices $A \in \Mat_n(\F)$ there exists an integer $q \geq 0$ such that $(A+\lambda I_n)\oplus \lambda I_q$ is a
$(p_1,p_2,p_3)$-sum, a problem that is open for general values of $(p_1,p_2,p_3)$.
Nevertheless, for very specific values of $(p_1,p_2,p_3)$ we shall obtain a complete characterization.
Our first complete result deals with the case when $p_1=p_2=p_3=t^2$, i.e.\ we will completely characterize
the endomorphisms that split into the sum of three square-zero endomorphisms:

\begin{theo}\label{3squarezero}
Let $u \in \End(V)$, where $V$ is an infinite-dimensional vector space over $\F$.
Then, $u$ is the sum of three square-zero endomorphisms if and only if none of the following situations occurs:
\begin{enumerate}[(i)]
\item There exists $\lambda \in \F$ together with a finite-rank endomorphism $w \in \End(V)$ such that $\tr w \not\in \{0,\lambda\}$ and $u=\lambda\,\id_V+w$.
    \item The characteristic of $\F$ differs from $2$ and $u$ has a non-zero dominant eigenvalue.
\end{enumerate}
\end{theo}

Below, we rewrite this result by discussing whether the ground field has characteristic $2$ or not.

\begin{cor}
Let $V$ be an infinite-dimensional vector space over a field $\F$ with characteristic different from $2$,
and let $u \in \End(V)$.
Then, $u$ is the sum of three square-zero endomorphisms if and only if it satisfies none of the following conditions:
\begin{enumerate}[(i)]
\item $u$ has finite rank and non-zero trace;
\item $u$ has a non-zero dominant eigenvalue.
\end{enumerate}
\end{cor}

\begin{cor}
Let $V$ be an infinite-dimensional vector space over a field $\F$ with characteristic $2$, and let $u \in \End(V)$.
Then, $u$ is the sum of three square-zero endomorphisms if and only if
there is no scalar $\lambda$ such that $u-\lambda\id_V$ has finite rank and trace different from $0$ and $\lambda$.
\end{cor}

Our second special case is the one when $p_1=p_2=p_3=t^2-t$, over fields with characteristic $2$.

\begin{theo}\label{3idemcar2theo}
Let $V$ be an infinite-dimensional vector space over a field $\F$ with characteristic $2$,
and let $u \in \End(V)$. Then, $u$ is the sum of three idempotent endomorphisms of $V$
if and only if none of the following conditions holds:
\begin{enumerate}[(i)]
\item $u$ has a dominant eigenvalue outside of $\{0_\F,1_\F\}$;
\item There exists $\lambda \in \{0_\F,1_\F\}$ such that $u-\lambda\,\id_V$ has finite rank and trace outside of $\{0_\F,1_\F\}$.
\end{enumerate}
\end{theo}

Our final result generalizes one that was known over finite-dimensional spaces \cite{Rabanovich,dSPLC3}:

\begin{theo}\label{LC3}
Let $V$ be a vector space over a field.
Then, every endomorphism of $V$ is a linear combination of three idempotent endomorphisms.
\end{theo}

\subsection{Some basic remarks and notation}

Following the French convention, we denote by $\N$ the set of all non-negative integers, by $\N^*$ the set of all positive ones,
and we use the word ``countable" to mean ``infinite countable", and ``uncountable" to mean ``infinite uncountable".
Throughout the article, $t$ denotes an indeterminate, $\F[t]$ the algebra of polynomials in the indeterminate $t$ and,
given a non-negative integer $d$, we denote by $\F_d[t]$ the linear subspace of $\F[t]$ consisting of all polynomials
with degree at most $d$.

Throughout the article, we will make frequent use of the following basic remarks.

\begin{Rem}\label{directsumremark}
Let $u$ be an endomorphism of a vector space $V$. Let $p_1,\dots,p_r$ be polynomials in $\F[t]$.
Assume that $V$ splits into $V=\underset{i \in I}{\bigoplus} V_i$ in which each $V_i$ is stable under $u$
and the resulting endomorphism is denoted by $u_i$.
Assume that, for all $i \in I$, the endomorphism $u_i$ splits into $u_i=\underset{k=1}{\overset{r}{\sum}} u_{i,k}$, where $u_{i,k}\in \End(V)$ and $p_k(u_{i,k})=0$ for all $k \in \lcro 1,r\rcro$.
Then, by setting $u^{(k)}:=\underset{ i \in I}{\bigoplus} u_{i,k}$ for all $k \in \lcro 1,r\rcro$, we see that
$u=\underset{k=1}{\overset{r}{\sum}} u^{(k)}$ and $p_k(u^{(k)})=0$ for all $k \in \lcro 1,r\rcro$.
\end{Rem}

\begin{Rem}[Reduction to the monic case]\label{monicremark}
Let $p$ be a non-zero polynomial, with leading coefficient $\lambda$.
An endomorphism is annihilated by $p$ if and only if it is annihilated by $\lambda^{-1} p$, a polynomial which has
the same degree and trace as $p$, and is split if and only if $p$ is split. Hence, in the proof of all the above theorems,
it will suffice to consider the case when the polynomials under consideration are all monic.
\end{Rem}

\begin{Rem}[The canonical situation]\label{canonicalremark}
In Theorem \ref{theo3}, no generality is lost in assuming that
each polynomial under consideration is of the form $t^2-at$ for some $a \in \F$.
Indeed, let $p_1,p_2,p_3$ be split polynomials with degree $2$ over $\F$.
Given $k \in \{1,2,3\}$, we denote the roots of $p_k$ by $x_k,y_k$. An endomorphism $v$ is annihilated by $p_k$
if and only if $q_k:=t^2-(y_k-x_k)t$ annihilates $v-x_k\id$.
It follows that an endomorphism $u \in \End(V)$ is a $(p_1,p_2,p_3)$-sum
if and only if $u-(x_1+x_2+x_3)\,\id_V$ is a $(q_1,q_2,q_3)$-sum.
Moreover, it is obvious that $u-(x_1+x_2+x_3)\,\id_V$ has a dominant eigenvalue if and only if $u$ does have one.
In particular, in Theorem \ref{theo3} it will suffice to consider the situation where each
$p_i$ equals $t^2-a_i t$ for some scalar $a_i$.
\end{Rem}

\subsection{Strategy, and structure of the article}

Our main strategy for the proof of Theorem \ref{theo3} is globally similar
to the one that was used in \cite{dSPSum4}.
First of all, let $u$ be an endomorphism of a vector space $V$ over $\F$. The vector space structure of $V$
is enriched into an $\F[t]$-module $V^u$ by setting $t\,x:=u(x)$ for all $x \in V$. We say that $u$ is \textbf{elementary}
when $V^u$ is a free $\F[t]$-module.
A basic result that was proved in \cite{dSPSum4} (Theorem 1 in that article) reads as follows:

\begin{theo}\label{theo2}
Let $u$ be an elementary endomorphism of a vector space $V$, and $p_1,p_2$ be split polynomials
with degree $2$ over $\F$. Then, $u$ is a $(p_1,p_2)$-sum.
\end{theo}

In \cite{dSPSum4}, the basic strategy then consisted in showing that, given split polynomials $p_3,p_4$
with degree $2$ over $\F$ and an endomorphism $u$ of $V$, there exist endomorphisms $u_3$ and $u_4$
of $V$ such that $u-u_3-u_4$ is elementary and $p_3(u_3)=p_4(u_4)=0$
(note that in \cite{Shitov}, Shitov proved the more powerful result that
$u$ is actually the sum of two elementary endomorphisms).
Here, the strategy will be to start from an endomorphism $u$ of $V$
and from a split polynomial $p$ with degree $2$ over $\F$, and to search for
an endomorphism $v$ of $V$ such that $u-v$ is elementary and $p(v)=0$.
A definition is relevant here:

\begin{Def}
Let $V$ be an $\F$-vector space, $u$ be an endomorphism of $V$, and $a$ be a scalar.
We say that $u$ is \textbf{$a$-elementarily decomposable} whenever there exists
an endomorphism $v$ of $V$ such that $v^2=a v$ and $u-v$ is elementary.
\end{Def}

Combining Theorem \ref{theo2} with Remark \ref{canonicalremark}, one sees that, in order to prove Theorem \ref{theo3}, it only remains to establish the following result:

\begin{theo}\label{alphaelementarilyTheo}
Let $u$ be an endomorphism of an infinite-dimensional vector space $V$, with no dominant eigenvalue.
Then, for any scalar $a$, the endomorphism $u$ is $a$-elementarily decomposable.
\end{theo}

In order to find a well-suited $v$, we shall use similar methods as in \cite{dSPSum4}: they involve  \emph{stratifications}
of the $\F[t]$-module $V^u$ and \emph{connectors}.
We will review the relevant definitions and results on them in Section \ref{stratsection}, and
we will also give new results that are needed here. The key notion is the one of a \emph{good stratification}, that is defined in Section \ref{goodstratsection}.

The article is laid out as follows. Section \ref{CNsection} essentially consists of a proof of
Theorem \ref{dominanteigenvalueCN} but also includes a technical lemma (the invariant subspace lemma)
that will be used in later sections, together with a characterization of the scalar multiples of the identity that
are $(p_1,p_2,p_3)$-sums. In Section \ref{finiteranksection}, we will discuss what should be done in general to tackle the case of the sum of a scalar multiple of the identity with a finite-rank endomorphism.
In Section \ref{sufficientdominantsection}, we will derive Theorem \ref{dominanteigenvalueCS}
from Theorem \ref{theo3} and from the characterization of the $(p_1,p_2,p_3)$-sums among the scalar multiples of the identity.

The rest of the article is mainly devoted to the proof of Theorem \ref{alphaelementarilyTheo}.
Section \ref{stratsection} consists of a discussion of stratifications and connectors.
In Section \ref{uncountablesection}, we will prove Theorem \ref{alphaelementarilyTheo} in the special case of a vector space with uncountable dimension. Section \ref{countablesection} deals with the difficult case of a vector space with countable dimension.
In the final section (Section \ref{specialcasessection}), we shall complete the study by proving Theorems \ref{3squarezero},
\ref{3idemcar2theo} and \ref{LC3}, with the help of recent results on the finite-dimensional case \cite{dSP3squarezero}.

\section{Necessary conditions}\label{CNsection}

\subsection{The case of scalar multiples of the identity}

\begin{prop}\label{homothetieprop}
Let $V$ be an infinite-dimensional vector space over $\F$, and $\lambda \in \F$.
Let $p_1,p_2,p_3$ be split polynomials with degree $2$ over $\F$.
Then, $\lambda\,\id_V$ is a $(p_1,p_2,p_3)$-sum if and only if
$\lambda$ is a $(p_1,p_2,p_3)$-sum or $2\lambda=\tr p_1+\tr p_2+\tr p_3$.
\end{prop}

\begin{proof}
By Remark \ref{monicremark}, we lose no generality in assuming that $p_1,p_2,p_3$ are all monic,
and we shall denote their respective traces by $\alpha,\beta,\gamma$.

We start with the converse implication. If $\lambda$ is a $(p_1,p_2,p_3)$-sum, we split it up into
$\lambda=x_1+x_2+x_3$ where $p_i(x_i)=0$ for all $i \in \lcro 1,3\rcro$ and it is then obvious
from writing $\lambda\,\id_V=x_1\,\id_V+x_2\,\id_V+x_3\,\id_V$ that $\lambda\,\id_V$ is a $(p_1,p_2,p_3)$-sum.
Assume now that $2\lambda=\alpha+\beta+\gamma$. Write $p_1(t)=(t-x)(t-y)$ with $(x,y)\in \F^2$.
It is obvious that we can find scalars $\mu$ and $\nu$ such that the matrices
$B:=\begin{bmatrix}
0 & \mu \\
1 & \beta
\end{bmatrix}$ and $C:=\begin{bmatrix}
\lambda-x & \nu \\
-1 & \gamma+x-\lambda
\end{bmatrix}$ are respectively annihilated by $p_2$ and $p_3$
(one simply chooses $\mu$ and $\nu$ such that the determinants of those matrices are, respectively, $p_2(0)$ and $p_3(0)$,
and one concludes thanks to the Cayley-Hamilton theorem).
Then, since $2\lambda=x+y+\beta+\gamma$, we have
$$\lambda\,I_2-B-C=\begin{bmatrix}
x & -\mu-\nu \\
0 & y
\end{bmatrix}.$$
Once more, the Cayley-Hamilton theorem yields that $A:=\lambda\,I_2-B-C$ is annihilated by $p_1$.
Hence, $\lambda\,I_2$ is a $(p_1,p_2,p_3)$-sum. It follows that, for every $2$-dimensional vector space $P$,
the endomorphism $\lambda\,\id_P$ is a $(p_1,p_2,p_3)$-sum.

Now, since $V$ is infinite-dimensional we can split it up as $V=\underset{i \in I}{\bigoplus} P_i$
where each $P_i$ is a $2$-dimensional vector space. Then, $\lambda\,\id_{P_i}$ is a $(p_1,p_2,p_3)$-sum for all $i \in I$, and
by Remark \ref{directsumremark} we conclude that $\lambda\,\id_V$ is a $(p_1,p_2,p_3)$-sum.

\vskip 3mm
We now turn to the direct implication. Assume that $\lambda\,\id_V=a+b+c$ for some triple $(a,b,c)\in \End(V)^3$
such that $p_1(a)=0$, $p_2(b)=0$ and $p_3(c)=0$. Assume also that $2\lambda \neq \alpha+\beta+\gamma$.
Then, we shall prove that $\lambda$ is a $(p_1,p_2,p_3)$-sum.
Note first that $b$ and $c$ commute with $u:=(b+c)\bigl((\beta+\gamma)\,\id_V-b-c\bigr)$ (see Lemma 3 of \cite{dSPsumoftwotriang}).
Indeed, by expanding we get that $u=\gamma b +\beta c-bc-cb+\delta \id_V$ for some $\delta \in \F$,
and one checks that $b$ commutes with $\beta c-bc-cb$: indeed,
$b(\beta c-bc-cb)=\beta bc-b^2 c-bcb=p_2(0)c-bcb$ and likewise $(\beta c-bc-cb)b=p_2(0)c-bcb$; it follows that
$b$ commutes with $u$, and likewise $c$ commutes with $u$.

Next, we use $b+c=\lambda\,\id_V-a$ to obtain
$$u=(\lambda\,\id_V-a)\bigl((\beta+\gamma-\lambda)\,\id_V+a\bigr).$$
By expanding and using $a^2 \in \alpha\,a+\F \id_V$, we get that
$$u=(2\lambda-\alpha-\beta-\gamma)\,a+\delta' \,\id_V$$
for some $\delta' \in \F$. Since $2\lambda-\alpha-\beta-\gamma \neq 0$, we deduce that $b$ and $c$ both commute with $a$.

Symmetrically, we obtain that $b$ commutes with $c$.
Now, classically since $a,b,c$ are pairwise commuting endomorphisms of a non-zero vector space that are annihilated by split polynomials,
they have a common eigenvector; denoting by $x,y,z$ the corresponding eigenvalue for $a,b,c$, respectively,
we deduce that $\lambda=x+y+z$, which shows that $\lambda$ is a $(p_1,p_2,p_3)$-sum.
\end{proof}

\subsection{The invariant subspace lemma}

\begin{lemma}[Invariant subspace lemma]\label{invariantsubspacelemma}
Let $V$ be an infinite-dimensional vector space over $\F$.
Let $a,b,c$ be quadratic endomorphisms of $V$. Let $\lambda \in \F$ and $w \in \End(V)$ be such that:
\begin{enumerate}[(i)]
\item $\lambda\,\id_V+w=a+b+c$;
\item $\rk w <\dim V$.
\end{enumerate}
Let $W$ be a linear subspace of $V$ that includes $\im w$ and such that $\dim W<\dim V$.
Then, there exists a linear subspace $\overline{W}$ of $V$ such that $\dim \overline{W}<\dim V$,
$\overline{W}$ includes $W$ and is stable under $a$, $b$ and $c$.
Moreover if $W$ is finite-dimensional then $\overline{W}$ can be chosen finite-dimensional.
\end{lemma}

\begin{proof}
We set
$$\overline{W}:=W+a(W)+b(W)+c(W)+\sum_{(e,f)\in \{a,b,c\}^2} (ef)(W).$$
Each vector space in this sum has dimension less than or equal to $\dim W$.
Hence, $\dim \overline{W}<\dim V$ because $V$ is infinite-dimensional. Moreover if $W$ is finite-dimensional then so is $\overline{W}$.

Since $\overline{W}$ includes $W$, it only remains to show that $\overline{W}$ is stable under $a,b,c$.
Obviously, it suffices to show that $(efg)(W) \subset \overline{W}$ for all $e,f,g$ in $\{a,b,c\}$.
Note first that if $e=f$, then $efg \in \Vect(eg,g)$ since $e$ is quadratic, whence
$(efg)(W) \subset \overline{W}$. Likewise, this inclusion also holds if $f=g$.

Next,
\begin{align*}
ab+ba & =(a+b)^2-a^2-b^2=(\lambda\,\id_V+w-c)^2-a^2-b^2 \\
& =\lambda^2\,\id_V+2\lambda w+w^2-2\lambda c-wc-cw+c^2-a^2-b^2.
\end{align*}
Since $a,b,c$ are all quadratic and $W$ includes $\im w$, it follows that
$$(ab+ba)(W) \subset W+a(W)+b(W)+c(W).$$
Next, $aba=a(ab+ba)-a^2b \in \Vect(a(ab+ba),ab,b)$, and we deduce from this equality and from the previous inclusion that
$$(aba)(W) \subset a(W)+a^2(W)+(ab)(W)+(ac)(W)+(ab)(W)+b(W).$$
Since $a^2$ is quadratic, $a^2(W) \subset a(W)+W$, and hence
$$(aba)(W) \subset \overline{W}.$$
More generally, we obtain $(efe)(W) \subset \overline{W}$ for all $(e,f) \in \{a,b,c\}^2$ (the case when $e=f$ has already been dealt with).
Finally,
\begin{align*}
cba=(\lambda \id_V+w-a-b)(ba)& =\lambda (ba)+w(ba)-(aba)-b^2 a \\
& \in \Vect(ba,w(ba),aba,b^2a).
\end{align*}
We have already seen that the image of $W$ under each endomorphism $ba,w(ba),aba,b^2a$ is included in $\overline{W}$,
whence $(cba)(W) \subset \overline{W}$.
Hence, symmetrically $(efg)(W) \subset \overline{W}$ for all distinct $e,f,g$ in $\{a,b,c\}$.
We conclude that $\overline{W}$ is stable under $a$, $b$ and $c$, which completes the proof.
\end{proof}

\subsection{Proof of Theorem \ref{dominanteigenvalueCN}}

Here, we derive Theorem \ref{dominanteigenvalueCN} from the preceding two results.
Let $p_1,p_2,p_3$ be split polynomials of $\F[t]$ with degree $2$,
and let $u$ be an endomorphism of an infinite-dimensional vector space $V$.
Assume that $u$ has a dominant eigenvalue $\lambda$ and that there exist endomorphisms $a,b,c$ of $V$
such that $u=a+b+c$ and $p_1(a)=p_2(b)=p_3(c)=0$.

Set $w:=u-\lambda\,\id_V$.
By Lemma \ref{invariantsubspacelemma} applied to $W:=\im w$, there exists a linear subspace $\overline{W}$ of
$V$ that includes $\im w$, is stable under $a$, $b$ and $c$, and whose dimension is less than the one of $V$.
It follows that $a,b,c$ induce endomorphisms of the infinite-dimensional quotient space $\overline{V}:=V/\overline{W}$ whose
sum equals $\lambda\id_{\overline{V}}$.
By Proposition \ref{homothetieprop}, we deduce that $\lambda$ is a $(p_1,p_2,p_3)$-sum or
$2\lambda=\tr p_1+\tr p_2+\tr p_3$. This completes the proof of Theorem \ref{dominanteigenvalueCN}.

\section{The case of the sum of a scalar multiple of the identity with a finite-rank endomorphism}\label{finiteranksection}

In this section, we shall give a partial result to the problem of determining when
an endomorphism of the form $\lambda \id_V+w$, where $\lambda$ is a scalar and $w$ is a finite-rank endomorphism,
is a $(p_1,p_2,p_3)$-sum.

The following definition is relevant to this problem:

\begin{Def}
Let $A$ be an $n$-by-$n$ matrix with entries in $\F$, and let $\lambda \in \F$.
Let $p_1,p_2,p_3$ be split polynomials with degree $2$ over $\F$.
Let $\lambda \in \F$ be such that $2\lambda=\tr p_1+\tr p_2+\tr p_3$ or $\lambda$ is a $(p_1,p_2,p_3)$-sum.
We say that $A$ is \textbf{a $(p_1,p_2,p_3)$-sum $\lambda$-stably} if
there exists a non-negative integer $q$ such that the block-diagonal matrix $(A+\lambda I_n) \oplus \lambda I_q$ is a $(p_1,p_2,p_3)$-sum.
\end{Def}

Next, to any finite-rank endomorphism $w$ of $V$ can be attached a similarity class of square matrices as follows:
we choose a minimal (finite-dimensional) linear subspace $W$ of $V$ such that $\im w \subset W$ and $W+\Ker w=V$.
The dimension of $W$ does not depend on the specific choice of $W$ and we denote it by $n(w)$.
Then, the similarity class of matrices that is attached to the induced endomorphism $w_{|W}$ does not depend on the choice
of $W$ either. We denote this similarity class by $[w]$. Moreover, if $W'$ is an arbitrary finite-dimensional linear subspace of
$V$ such that $\im w \subset W'$ and $W'+\Ker w=V$, any matrix that represents $w_{|W'}$
has the form $M \oplus 0_q$ for some $M \in [w]$ and some non-negative integer $q$.

Here is our partial result:

\begin{theo}\label{theofiniterank}
Let $V$ be an infinite-dimensional vector space,
$w$ be a finite-rank endomorphism of $V$ and $\lambda$ be a scalar.
Let $p_1,p_2,p_3$ be split polynomials with degree $2$ over $\F$.
Choose a matrix $A$ in $[w]$.
Then, the following conditions are equivalent:
\begin{enumerate}[(i)]
\item The endomorphism $\lambda \id_V+w$ is a $(p_1,p_2,p_3)$-sum.
\item The matrix $A$ is a $(p_1,p_2,p_3)$-sum $\lambda$-stably, and
either $\lambda$ is a $(p_1,p_2,p_3)$-sum or $2\lambda=\tr p_1+\tr p_2+\tr p_3$.
\end{enumerate}
\end{theo}

\begin{proof}
Set $u:=\lambda \id_V+w$.

Assume first that condition (ii) holds.
Choose a non-negative integer $q$ such that
$(A +\lambda I_{n(w)}) \oplus \lambda I_q$ is a $(p_1,p_2,p_3)$-sum.
We have a finite-dimensional linear subspace $W$ of $V$ such that $\im w \subset W$, $W+\Ker w=V$,
and $A$ represents the endomorphism of $W$ induced by $w$.
Then, we can split $V=W \oplus W'$ where $W' \subset \Ker w$.
We can further split $W'=W'_1 \oplus W'_2$ so that $W'_1$ has dimension $q$.
Hence, $V=W \oplus W'_1 \oplus W'_2$ and $W'_2$ is infinite-dimensional.
Since $(A \oplus \lambda I_{n(w)}) \oplus \lambda I_q$ is a $(p_1,p_2,p_3)$-sum,
the endomorphism $u_{|W \oplus W'_1}$ is a $(p_1,p_2,p_3)$-sum.
Moreover, by Proposition \ref{homothetieprop}, the
endomorphism $\lambda \id_{W'_2}$ of $W'_2$ is a $(p_1,p_2,p_3)$-sum.
Since $u_{|W'_2}=\lambda \id_{W'_2}$, we deduce from Remark \ref{directsumremark} that
$u$ is a $(p_1,p_2,p_3)$-sum.

Conversely, assume that condition (i) holds.
By Theorem \ref{dominanteigenvalueCN}, we already know that $\lambda$ is a $(p_1,p_2,p_3)$-sum or
$2\lambda=\tr p_1+\tr p_2+\tr p_3$.

Let $a,b,c$ be endomorphisms of $V$ such that
$u=a+b+c$ and $p_1(a)=p_2(b)=p_3(c)=0$.
Choosing a complementary subspace $G$ of $\Ker w$ in $V$ and applying the invariant subspace lemma to the finite-dimensional subspace
$W:=\im w+G$, we obtain a finite-dimensional linear subspace $W'$
of $V$ that is stable under $a,b,c$, includes $\im w$ and satisfies $W'+\Ker w=V$.
Then, we can split $V=W' \oplus V'$ such that $V' \subset \Ker w$.
Choosing a matrix $A$ in $[w]$, it follows that, for some non-negative integer $q$, the matrix
$(A+\lambda I_{n(w)}) \oplus (\lambda I_q)$ represents the endomorphism of $W'$ induced by $u$.
Since $a,b,c$ stabilize $W'$, this endomorphism turns out to be a $(p_1,p_2,p_3)$-sum, whence $(A+\lambda I_{n(w)}) \oplus (\lambda I_q)$
is a $(p_1,p_2,p_3)$-sum, and we conclude that $A$ is a $(p_1,p_2,p_3)$-sum $\lambda$-stably.
\end{proof}

Hence, in order to detect the $(p_1,p_2,p_3)$-sums among the
endomorphisms of type $\lambda \id_V+w$, with $w$ of finite rank, it remains to understand
which square matrices over $\F$ are $(p_1,p_2,p_3)$-sums $\lambda$-stably.
For general values of $p_1,p_2,p_3$, the latter problem is open, and probably intractable. For specific values of $(p_1,p_2,p_3)$, the
recent \cite{dSP3squarezero} provides some answers.

\section{Deriving Theorem \ref{dominanteigenvalueCS} from Theorem \ref{theo3}}\label{sufficientdominantsection}

\begin{lemma}[The reduction lemma]\label{reductionlemma}
Let $u \in \End(V)$, where $V$ is an infinite-dimensional vector space over $\F$.
Assume that $u$ has a dominant eigenvalue $\lambda$ and that $u-\lambda \id_V$ has infinite rank.
Then, there exists a decomposition $V=V_1 \oplus V_2$
into linear subspaces that are stable under $u$ and such that:
\begin{enumerate}[(i)]
\item $V_1$ is infinite-dimensional.
\item $u_{|V_1}$ has no dominant eigenvalue.
\item $u(x)=\lambda x$ for all $x \in V_2$.
\end{enumerate}
\end{lemma}

\begin{proof}
Replacing $u$ with $u-\lambda\id_V$, no generality is lost in assuming that $\lambda=0$.
Denote by $\rk u$ the rank of $u$.

We choose a complementary subspace $W$ of $\Ker(u)$ in $V$.
By the rank theorem, $W$ has dimension $\rk u$, which is less than $\dim V$
and hence than $\dim \Ker(u)$. Hence, we can also choose a linear subspace $W'$ of $\Ker u$ such that
$\dim W'=\rk u$ and $W' \cap (W+ \im u)=\{0\}$.
Set $V_1:=W'\oplus (W+ \im(u))$. The linear subspace $V_1$ is stable under $u$ since it includes $\im u$.
As $V_1+\Ker(u)=V$, we can choose a linear subspace $V_2$ of $\Ker u$ such that $V_1 \oplus V_2=V$.
Hence, $V_2$ is stable under $u$
and $u(x)=\lambda\,x$ for all $x \in V_2$. It remains to show that $V_1$ has the claimed properties.

Denote by $u_1$ the endomorphism of $V_1$ induced by $u$.
Remember that $u$ has infinite rank. Note that $\dim V_1=\dim W'=\rk u$
and that $\im u=\im u_1$ since $u$ vanishes everywhere on $V_2$. Hence, $0$ is not a dominant eigenvalue of $u_1$.
Let $\alpha \in \F \setminus \{0\}$. Then, $\Ker(u_1-\alpha\,\id_{V_1}) \subset \im u_1$, and hence
the codimension of $\Ker(u_1-\alpha\,\id_{V_1})$ in $V_1$ is greater than or equal to the dimension of $W'$, which proves that $\alpha$
is not a dominant eigenvalue of $u_1$.

Therefore, $u_1$ has no dominant eigenvalue.
\end{proof}

From there, we can derive Theorem \ref{dominanteigenvalueCS} from Theorem \ref{theo3}. Assume
that Theorem \ref{theo3} is valid.
Let $p_1,p_2,p_3$ be split polynomials with degree $2$ over $\F$,
and let $u$ be an endomorphism of an infinite-dimensional vector space $V$.
Assume that $u$ has a dominant eigenvalue $\lambda$, that $u-\lambda \id_V$ has infinite rank
and that either $\lambda$ is a $(p_1,p_2,p_3)$-sum or $2\lambda=\tr p_1+\tr p_2+\tr p_3$.

We can find a decomposition $V=V_1\oplus V_2$ given by Lemma \ref{reductionlemma}.
Then, the endomorphism of $V_2$ induced by $u$ is a $(p_1,p_2,p_3)$-sum, owing to Proposition \ref{homothetieprop}.
On the other hand, the endomorphism $u_1$ of $V_1$ induced by $u$ has no dominant eigenvalue and $V_1$ is infinite-dimensional,
whence by Theorem \ref{theo3} the endomorphism $u_1$ is a $(p_1,p_2,p_3)$-sum. By Remark \ref{directsumremark}, we conclude that $u$ is a $(p_1,p_2,p_3)$-sum.

\newpage
\section{Stratifications}\label{stratsection}

Throughout this section, we shall need the following notation on well-ordered sets:

\begin{Def}
Let $D$ be a well-ordered set and $\alpha$ be an element of $D$, but not the greatest one.
Then, we denote by $\alpha+1$ the \textbf{successor} of $\alpha$ (that is, the least element of $\{\beta \in D : \alpha<\beta\}$).
We say that $\alpha$ has a predecessor whenever $\alpha$ is the successor of some element of $D$.
\end{Def}

\subsection{A review of known results}

\begin{Def}
Let $V$ be a non-zero $\F[t]$-module.
A \textbf{stratification} of $V$ is an increasing sequence $(V_\alpha)_{\alpha \in D}$, indexed over a well-ordered set $D$,
of submodules of $V$ in which:
\begin{itemize}
\item For all $\alpha \in D$, the quotient module $V_\alpha/\Bigl(\underset{\beta<\alpha}{\sum} V_\beta\Bigr)$
is non-zero and monogenous;
\item $V=\underset{\alpha \in D}{\sum} V_\alpha$.
\end{itemize}
To any such stratification, we assign the \textbf{dimension sequence} $(n_\alpha)_{\alpha \in D}$ defined by $$n_\alpha:=\dim_\F\biggl(V_\alpha/\underset{\beta<\alpha}{\sum} V_\beta\biggr)$$
(if this dimension is not finite, we shall denote it by $+\infty$ rather than by $\aleph_0$).
\end{Def}

Let $(V_\alpha)_{\alpha \in D}$ be a stratification of $V$.
For every $\alpha \in D$, we can choose a vector $x_\alpha \in V_\alpha$ such that $V_\alpha=\F[t] x_\alpha+\underset{\beta<\alpha}{\sum} V_\beta$,
and we note that if $n_\alpha$ is finite then
$V_\alpha=\F_{n_\alpha-1}[t] x_\alpha \oplus \underset{\beta<\alpha}{\sum} V_\beta$,
otherwise $V_\alpha=\F[t]\,x_\alpha \oplus \underset{\beta<\alpha}{\sum} V_\beta$, and in any case
$(t^k x_\alpha)_{0 \leq k<n_\alpha}$ is linearly independent over $\F$.
We shall say that the \textbf{vector sequence} $(x_\alpha)_{\alpha \in D}$ is attached to $(V_\alpha)_{\alpha \in D}$.
In that case, an obvious transfinite induction shows that, for all $\alpha$ and $\beta$ in $D$ with $\beta<\alpha$,
the family $(t^k\,x_\delta)_{\beta \leq \delta \leq \alpha, \; 0 \leq k <n_\delta}$ is linearly independent over $\F$ and
$$\biggl[\underset{\gamma<\beta}{\sum} V_\gamma\biggr] \oplus \Vect_\F\Bigl(\bigl(t^k\, x_\delta\bigr)_{\beta \leq \delta \leq \alpha, \; 0 \leq k <n_\delta}\Bigr)=V_\alpha.$$
Moreover the family $(t^k x_\delta)_{\beta \leq \delta < \alpha, \; 0 \leq k <n_\delta}$ is linearly independent over $\F$ and
$$\biggl[\underset{\gamma<\beta}{\sum} V_\gamma\biggr] \oplus \Vect_\F\Bigl(\bigl(t^k\,x_\delta\bigr)_{\beta \leq \delta < \alpha, \; 0 \leq k <n_\delta}\Bigr)
=\underset{\gamma<\alpha}{\sum} V_\gamma.$$
In particular, $(t^k\, x_\alpha)_{\alpha \in D, \; 0 \leq k <n_\alpha}$ is a basis of the vector space $V$.
As a special case, we get the obvious consequence:

\begin{lemma}
Let $V$ be an $\F[t]$-module with a stratification $(V_\alpha)_{\alpha \in D}$ whose corresponding dimension sequence
is denoted by $(n_\alpha)_{\alpha \in D}$. Assume that $n_\alpha=+\infty$ for all $\alpha \in D$.
Then, $V$ is free.
\end{lemma}

Conversely, consider a sequence $(x_\alpha)_{\alpha \in D}$, indexed over a well-ordered set $D$,
of vectors of $V$ such that $x_\alpha \not\in \underset{\beta <\alpha}{\sum} \F[t] x_\beta$ for all
$\alpha \in D$, and $V=\underset{\alpha \in D}{\sum} \F[t]x_\alpha$.
Then, one sees that $\Bigl(\underset{\beta \leq \alpha}{\sum} \F[t]\, x_\beta\Bigr)_{\alpha \in D}$
is a stratification of $V$ with corresponding vector sequence $(x_\alpha)_{\alpha \in D.}$

Let us now recall the definition of a connector.

\begin{Def}\label{defconnector}
Let $u$ be an endomorphism of a vector space $V$.
Let $(V_\alpha)_{\alpha \in D}$ be a stratification of $V^u$, with attached dimension sequence
$(n_\alpha)_{\alpha\in D}$ and an associated vector sequence $(x_\alpha)_{\alpha \in D.}$

An endomorphism $v$ of $V$ is called a \textbf{connector} for $u$ with respect to the vector sequence
$(x_\alpha)_{\alpha \in D}$ whenever it acts as follows on the basis $(t^k\,x_\alpha)_{\alpha \in D, 0 \leq k<n_\alpha}$:
for all $\alpha \in D$ such that $n_\alpha<+\infty$ and $\alpha$ is not the greatest element of $D$,
we have $v(t^{n_{\alpha}-1}\,x_\alpha)=x_{\alpha+1}$ modulo $V_\alpha$,
and all the other vectors are mapped to $0$.
\end{Def}

Here is the basic result that demonstrates the interest of connectors:

\begin{prop}[Proposition 8 of \cite{dSPSum4}]\label{connectorprop}
Let $u$ be an endomorphism of a vector space $V$.
Let $(V_\alpha)_{\alpha \in D}$ be a stratification of $V^u$, with attached dimension sequence
$(n_\alpha)_{\alpha\in D}$ and an associated vector sequence $(x_\alpha)_{\alpha \in D}$.

Assume that if $D$ has a maximum $M$ then $n_M=+\infty$. Then, for any connector
$v$ for $u$ with respect to $(x_\alpha)_{\alpha \in D}$, the endomorphism $u+v$ is elementary.
\end{prop}

\subsection{Good stratifications}\label{goodstratsection}

In light of Proposition \ref{connectorprop}, our wish is to create a connector that is annihilated by
$t^2-at$ for a given scalar $a$. This is possible if we consider special cases of stratifications:

\begin{Def}
Let $(V_\alpha)_{\alpha \in D}$ be a stratification of $V$, with corresponding dimension sequence $(n_\alpha)_{\alpha \in D}$.
We define three potential properties of that stratification:
\begin{itemize}
\item[$\PA$] One has $n_\alpha \geq 2$ whenever $\alpha \in D$ has a predecessor.
\item[$\PAplus$] One has $n_\alpha \geq 2$ whenever $\alpha \in D$ has a predecessor or $\alpha$ is the minimum of $D$.
\item[$\PM$] There is no maximum in $D$.
\end{itemize}
A stratification is called \textbf{good} whenever it satisfies both properties $\PAplus$  and $\PM$.
\end{Def}

The following basic result motivates that we focus on good stratifications:

\begin{prop}
Let $V$ be a vector space and $u$ be an endomorphism of $V$. Let $a \in \F$.
Let $(V_\alpha)_{\alpha \in D}$ be a stratification of $V^u$ that satisfies properties $\PA$
and $\PM$.
Then, $u$ is $a$-elementarily decomposable.
\end{prop}

\begin{proof}
Let $(x_\alpha)_{\alpha \in D}$ be a vector sequence attached to $(V_\alpha)_{\alpha \in D}$, and denote by
$(n_\alpha)_{\alpha \in D}$ the associated dimension sequence.
We define $v\in \End(V)$ on the basis $(u^k(x_\alpha))_{\alpha \in D, 0 \leq k<n_\alpha}$
as follows: For all $\alpha \in D$ such that $n_\alpha<+\infty$, we put
$$v\bigl(u^{n_\alpha-1}(x_\alpha)\bigr):=a u^{n_\alpha-1}(x_\alpha)-x_{\alpha+1},$$
(note that this makes sense because, by property $\PM$, the element $\alpha$ must have a successor)
and all the other basis vectors are mapped to $0$. Then, $-v$ is a connector for $u$
with respect to the sequence $(x_\alpha)_{\alpha \in D}$, and hence $u-v$ is elementary.
On the other hand, we check that $v^2=a v$: Given $\alpha \in D$ such that $n_\alpha<+\infty$,
we see that $v(x_{\alpha+1})=0$ because of property $\PA$, and it follows that
$v^2$ and $av$ agree on $u^{n_\alpha-1}(x_\alpha)$; on the other hand both $v^2$ and $av$ vanish at all the other
basis vectors, and hence $v^2=av$. This completes the proof.
\end{proof}

\begin{cor}\label{goodstrattoelementary}
Let $V$ be a vector space and $u$ be an endomorphism of $V$ with a good stratification.
Then, for all $a \in \F$, the endomorphism $u$ is $a$-elementarily decomposable.
\end{cor}

\subsection{A technical lemma on special stratifications}

The following technical result will be used in remote parts of the article.

\begin{lemma}\label{towerofstrat}
Let $V$ be a non-zero $\F[t]$-module. Assume that there is a non-zero submodule $W$ of $V$
such that $V/W$ has a good stratification and $W$ has a stratification that satisfies $\PAplus$. Then, $V$ has a good stratification.
\end{lemma}

\begin{proof}
We take a stratification $(W_k)_{k \in D}$ of $W$ that satisfies $\PAplus$ and a good stratification
$(V_l)_{l \in D'}$ of $V/W$, in which $D$ and $D'$ are ordinals.
We denote by $(n_k)_{k \in D}$ and $(m_l)_{l \in D'}$ the associated dimension sequences.
 We equip
 $$L:=(\{0\} \times D) \cup (\{1\} \times D')$$
with the lexicographic ordering, which makes it a well-ordered set.
For $k \in D$, we set $E_{0,k}:=W_k$ and, for $l \in D'$, we define $E_{1,l}$ as the inverse image of
$V_l$ under the canonical projection from $V$ to $V/W$.
In particular, $E_{0,k} \subset W \subsetneq E_{1,l}$ for all $k \in D$ and $l \in D'$, and it is easily checked
 that $(E_a)_{a \in L}$ is an increasing sequence of submodules of $V$.
Moreover, for all $k \in D$, we see that $\underset{a \in L, a<(0,k)}{\sum} E_a=\underset{k' \in D, k'<k}{\sum} W_{k'}$ and
hence the $\F[t]$-module $E_{0,k}/ \biggl(\underset{a \in L, a<(0,k)}{\sum} E_a\biggr)$ is monogenous and has dimension $n_k$ as an $\F$-vector space.
Given $l \in D'$, since $(E_a)_{a \in L}$ is increasing we see that
$\underset{a \in L, a<(1,l)}{\sum} E_a=W+\underset{l' \in D', l'<l}{\sum} E_{1,l'}$ which includes $W$, and hence
$$E_{1,l}/\biggl(\underset{a \in L, a<(1,l)}{\sum} E_a\biggr)=E_{1,l}/\biggl(W+\underset{l' \in D', l'<l}{\sum} E_{1,l'}\biggr) \simeq
V_l/\biggl(\underset{l' \in D', l'<l}{\sum} V_{l'}\biggr).$$
Therefore, the $\F[t]$-module $E_{1,l}/\biggl(\underset{a \in L, a<(1,l)}{\sum} E_a\biggr)$ is monogenous and has dimension $m_l$ as an $\F$-vector space.

Hence, $(E_a)_{a \in L}$ is a stratification of $V$. It remains to check that it is a good one.
If $L$ had a maximum, then this maximum would read $(1,M)$ and $M$ would be the maximum of $D'$, which is
impossible because $(V_l)_{l \in D'}$ has property $\PM$.

Finally, let $a \in L$ have a predecessor in $L$ or be the minimum of $L$. If $a=(0,k)$ for some $k \in D$, then $k$ has a predecessor in
$D$ or is the minimum of $D$, and we deduce that $n_k \geq 2$. If $a=(1,l)$ for some $l \in D'$ then we obtain likewise that
$m_l \geq 2$. We deduce that $(E_a)_{a \in L}$ has property $\PAplus$.
\end{proof}

\section{Endomorphisms with no dominant eigenvalue: The uncountable-dimensional case}\label{uncountablesection}

Here, we consider the case of a vector space with uncountable dimension.
In order to prove Theorem \ref{alphaelementarilyTheo} in that restricted setting,
we know from Corollary \ref{goodstrattoelementary} that it suffices to prove
the following result:

\begin{prop}
Let $V$ be a vector space with uncountable dimension, and $u$ be an endomorphism of $V$
with no dominant eigenvalue. Then, $V^u$ has a good stratification.
\end{prop}

\begin{proof}
Denote by $\kappa$ the dimension of $V$: it is a cardinal.
Since $\kappa$ is uncountable, the set $L$ consisting of the limit ordinals in $\kappa$
has cardinality $\kappa$, whence we can choose a basis $(e_\alpha)_{\alpha \in L}$ of the $\F$-vector space $V$.

The construction is now done by transfinite induction.
Let $\alpha \in \kappa$, and assume that we have constructed partial sequences
$(E_\beta)_{\beta< \alpha}$, $(x_\beta)_{\beta < \alpha}$ and $(n_\beta)_{\beta < \alpha}$ of, respectively,
linear subspaces, nonzero vectors, and elements of $\N^* \cup \{+\infty\}$, such that:
\begin{enumerate}[(i)]
\item For all $\beta <\alpha$, $E_\beta$ is a linear subspace of $V$
and $\bigl(u^k(x_\beta)\bigr)_{0 \leq k<n_\beta}$ is a basis of it;
\item The vector spaces $E_\beta$, for $\beta<\alpha$, are linearly disjoint;
\item For all $\beta <\alpha$, if $n_\beta <+\infty$ then $u^{n_\beta}(x_\beta) \in \underset{\gamma \leq \beta}{\bigoplus} E_\gamma$;
\item For every $\beta<\alpha$, if $\beta\not\in L$ then $n_\beta \geq 2$, otherwise $e_\beta \in \underset{\gamma \leq \beta}{\sum} E_\gamma$.
\end{enumerate}
Set $W:=\underset{\beta <\alpha}{\bigoplus} E_\beta$. By properties (i) and (iii), the linear subspace $W$ is stable under $u$.

We claim that the endomorphism $\overline{u}$ induced by $u$ on $V/W$ is not a scalar multiple of identity.
If it were, there would be a scalar $\lambda$ such that $\im(u-\lambda \id_V) \subset W$.
However, since $\kappa$ is a cardinal, $\{\beta \in \kappa: \beta<\alpha\}$ has its cardinality less than $\kappa$,
and since each $E_\beta$ has its dimension countable or finite, this yields $\dim W < \kappa$.
Hence, $\lambda$ would be a dominant eigenvalue of $u$, in contradiction with our assumptions.
By the classical characterization of the scalar multiples of the identity among the endomorphisms,
there exists a vector $y \in V$ such that $(y,u(y))$ is linearly independent modulo $W$.

Now, we put $x_\alpha:=e_\alpha$ if $\alpha \in L$ and $e_\alpha \not\in W$, otherwise $x_\alpha:=y$.
In any case, we take $n_\alpha$ as the supremum of the set of all $k \in \N$ for which $(u^i(x_\alpha))_{0 \leq i<k}$ is linearly independent modulo $W$,
and
$$E_\alpha:=\Vect\bigl(u^i(x_\alpha)\bigr)_{0 \leq i<n_\alpha.}$$
By the very definition of $x_\alpha$, we have $n_\alpha \geq 2$ if $\alpha \not\in L$, and
$x_\alpha \in \underset{\beta \leq \alpha}{\sum} E_\beta$ otherwise.
It is then easily checked that the spaces $E_\beta$, for $\beta\leq \alpha$,
are linearly disjoint and that if $n_\alpha<+\infty$ then $u^{n_\alpha}(x_\alpha) \in \underset{\beta \leq \alpha}{\bigoplus} E_\beta$.

The inductive step is now achieved.
By property (iv) above, the subspace $\underset{\beta \in \kappa}{\sum} E_\beta$ contains all the basis vectors $e_\alpha$ with $\alpha \in L$,
and hence $V=\underset{\beta \in \kappa}{\sum} E_\beta$. For $\alpha \in \kappa$, set $V_\alpha:=\underset{\beta \leq \alpha}{\bigoplus} E_\beta$.
Then, one sees from properties (i) to (iv) that $(V_\alpha)_{\alpha \in \kappa}$ is a good stratification of
$V^u$.
\end{proof}

To further illustrate the specificity of the uncountable-dimensional case, we
give an example when $V^u$ has no good stratification whereas $u$ has no dominant eigenvalue.

\begin{ex}\label{exampleofnogoodstrat}
Consider the $\F[t]$-module $V:=\F[t] \times \bigl(\F[t]/(t)\bigr)$,
and consider the vectors $e:=(1,0)$ and $f:=(0,1)$ in $V$, so that $V=\F[t]e \oplus \F[t] f$.
Assume that $V$ has a good stratification $(V_\alpha)_{\alpha \in D}$, and denote by $m$
the least element of $D$. Let $x$ be a generator of $V_m$.
Since $\dim V_m \geq 2$, we have $x \not\in \F[t] f$. Hence,
$x=p(t)\,e+\lambda f$ for some $p(t)\in \F[t] \setminus \{0\}$ and some $\lambda \in \F$.
The degree $d$ of $p(t)$ is non-negative, and it is easily checked that $V_m=\F[t] x$ does not contain $f$.
Hence, $m$ has a successor in $D$, which we denote by $m+1$.
Moreover, it is obvious that the respective classes of $f,e,t\,e,\dots,t^{d-1}e$ generate the vector space $V/V_m$, and hence
$V/V_m$ is a non-zero $\F$-vector space with finite dimension. Hence, $D$ must be finite, and $(V_\alpha)_{\alpha \in D}$
fails to be a good stratification since it does not satisfy condition $\PM$.

With a similar strategy, one can prove that $\F[t] \times (\F[t]/(t))^2$
has no stratification that satisfies both conditions $\PA$ and $\PM$.
\end{ex}

\section{Endomorphisms with no dominant eigenvalue: The countable-dimensional case}\label{countablesection}

In this section, we shall complete the proof of Theorem \ref{alphaelementarilyTheo} by tackling the special
case of vector spaces with countable dimension.
Here, the situation is far more complicated than the one of the preceding section because, for an endomorphism $u$ with no dominant eigenvalue,
the module $V^u$ might not have a good stratification (see Example \ref{exampleofnogoodstrat}).

We shall start by considering the case when $V^u$ is a torsion $\F[t]$-module,
and we will show in this situation (and still assuming that $u$ has no dominant eigenvalue)
that it must have a good stratification (Section \ref{torsioncase}).
In Section \ref{nontorsioncase}, we will complete the proof by tackling the case
when $V^u$ is not a torsion $\F[t]$-module, with the help of some results from the torsion case.

\subsection{The case of a torsion $\F[t]$-module}\label{torsioncase}

Our aim is to prove the following result:

\begin{prop}\label{torsiongoodstrat}
Let $u$ be an endomorphism of a vector space $V$ with countable dimension.
Assume that $u$ has no dominant eigenvalue and that $V^u$ is a torsion $\F[t]$-module.
Then, $V^u$ has a good stratification.
\end{prop}

Combining this result with Corollary \ref{goodstrattoelementary} obviously yields
Theorem \ref{alphaelementarilyTheo} in the special case when $V^u$ has countable dimension over $\F$ and is a torsion module.

For its proof, we need two lemmas.

\begin{lemma}\label{index2goodstrat}
Let $u$ be an endomorphism of a vector space $V$ with countable dimension.
Assume that $u$ has no dominant eigenvalue and that $(u-\lambda \id_V)^2=0$
for some $\lambda \in \F$. Then, $V^u$ has a good stratification.
\end{lemma}

\begin{proof}
Since $(u-\lambda \id_V)^2=0$,
there are families $(e_i)_{i \in I}$ and $(f_j)_{j \in J}$ of non-zero vectors
such that
$$V=\underset{i \in I}{\bigoplus} \Vect(e_i) \oplus \underset{j \in J}{\bigoplus} \Vect\bigl(f_j,u(f_j)\bigr),$$
with $e_i \in \Ker (u-\lambda \id_V)$ for all $i \in I$, and $f_j \not\in \Ker (u-\lambda \id_V)$ for all $j \in J$.

Since $\lambda$ is not a dominant eigenvalue of $u$, the set $J$ is infinite, and hence countable.
We choose a subset $A$ of $\N \setminus \{0\}$ with the same cardinality as $I$, and we put
$I':=A \times \{0\}$ and $J':=\N^2 \setminus I'$, so that $J'$ is equipotent to $J$.
Hence, without loss of generality we can assume that $(I,J)$ is a partition of $\N^2$ and $I \subset (\N \setminus \{0\})\times \{0\}$.
Then, for $(k,l) \in \N^2$, we put $z_{k,l}:=e_{k,l}$ if $(k,l) \in I$, and $z_{k,l}:=f_{k,l}$ otherwise.
The set $\N^2$ is well-ordered by the lexicographic ordering. Then, one
checks that the vector sequence $(z_{k,l})_{(k,l)\in \N^2}$ defines a good stratification of $V^u$.
\end{proof}

\begin{lemma}\label{twosubmoduleslemma}
Let $u$ be an endomorphism of a vector space $V$ with infinite dimension.
Assume that $u$ has no dominant eigenvalue, that
there is no scalar $\lambda$ such that $(u-\lambda \id_V)^2=0$, and that
$V^u$ is a torsion $\F[t]$-module.
Let $x_0$ be a vector of $V$.
Then, there are submodules $V_0$ and $V_1$ of $V^u$ such that:
\begin{enumerate}[(i)]
\item $V_0 \subset V_1$;
\item $V_1$ contains $x_0$;
\item $\dim_\F V_0 >1$ and $\dim_\F (V_1/V_0)>1$;
\item Each module $V_0$ and $V_1/V_0$ is monogenous.
\end{enumerate}
\end{lemma}

\begin{proof}
We distinguish between several cases. \\
\noindent \textbf{Case 1.} $x_0 \neq 0$ and $x_0$ is not an eigenvector of $u$. \\
Then, we set $V_0:=\F[t] x_0$, which is finite-dimensional as a vector space. Note that the endomorphism of $V/V_0$ induced by $u$
has no dominant eigenvalue, whence some non-zero vector $z \in V/V_0$ is not an eigenvector of it.
Denote by $V_1$ the inverse image of $\F[t]z$ under the canonical projection $V \rightarrow V/V_0$.
Then, $V_1/V_0=\F[t]z$ and $V_1/V_0$ has finite dimension greater than $1$ as a vector space over $\F$.

\noindent \textbf{Case 2.} $x_0=0$ or $x_0$ is an eigenvector of $u$. \\
Without loss of generality, we can assume that $u(x_0)=0$.

\noindent \textbf{Case 2.1.} The endomorphism $u$ is not locally nilpotent, i.e.\
we do not have $\forall x \in V, \; \exists n \in \N : \; u^n(x)=0$. \\
Since $V^u$ is a torsion module, there is a non-zero vector $y$ together with
a monic irreducible polynomial $p(t)\neq t$ such that $p(u)[y]=0$.
Then, $t$ and $p(t)$ are coprime, whence $V_0:=\F[t](x_0+y)$ contains $x_0$ and $y$,
and in particular $\dim_\F V_0 \geq 2$. Then, we find a submodule $V_1 \supset V_0$ as in Case 1.

\noindent \textbf{Case 2.2.} The endomorphism $u$ is locally nilpotent. \\
Our assumptions tell us that $u^2 \neq 0$.
This yields a vector $y \in V$ such that $u^2(y) \neq 0$ and $u^3(y)=0$.
Set $F:=\Vect(y,u(y),u^2(y),x_0)$.

\noindent \textbf{Case 2.2.1.} One has $x_0 \in \Vect\bigl(y,u(y),u^2(y)\bigr)$. \\
Then, we set $V_0:=\F[t] y=F$, and we construct $V_1$ as in Case 1.

\noindent \textbf{Case 2.2.2.} One has $x_0 \not\in \Vect\bigl(y,u(y),u^2(y)\bigr)$. \\
Then, we set $V_0:=\Vect\bigl(u(y)+x_0,u^2(y)\bigr)$ and $V_1:=F$.
Note that $V_0$ is the $\F[t]$-submodule generated by $u(y)+x_0$,
and that $(\overline{x_0},\overline{y})$ is a basis of the quotient space $V_1/V_0$.
Noting that $u(y)=-x_0$ modulo $V_0$, we see that the $\F[t]$-module $V_1/V_0$ is generated by $\overline{y}$.
Hence, $V_0$ and $V_1$ have the expected properties.
\end{proof}

Now, we can prove Proposition \ref{torsiongoodstrat}.

\begin{proof}[Proof of Proposition \ref{torsiongoodstrat}]
By a \emph{reductio ad absurdum}, we assume that $V^u$ has no good stratification.
We choose a basis $(e_n)_{n \in \N}$ of the $\F$-vector space $V$.

Then, by induction, we shall construct a good stratification of $V^u$.
Set $V_{-1}:=\{0\}$.
Let $n \in \N$, and assume that we have constructed a partial increasing sequence $(V_k)_{0 \leq k \leq 2n-1}$ of
submodules of $V^u$ such that $V_k/V_{k-1}$ is finite-dimensional over $\F$ with dimension greater than $1$ for all $k \in \lcro 0,n-1\rcro$,
and $V_{2k+1}$ contains $e_k$ for all $k \in \lcro 0,n-1\rcro$.

Then, we consider the quotient vector space $V/V_{2n-1}$ and the induced endomorphism $\overline{u}$ of it.
Since $V_{2n-1}$ is finite-dimensional as a vector space and $u$ has no dominant eigenvalue, $\overline{u}$ has no dominant
eigenvalue either. Moreover, $V/V_{2n-1}$ is a torsion $\F[t]$-module.

Assume first that there is a scalar $\lambda$ such that $(\overline{u}-\lambda \id)^2=0$.
Then, by Lemma \ref{index2goodstrat}, there is a good stratification of $V/V_{2n-1}$. However,
$(V_k)_{0 \leq k \leq 2n-1}$ is obviously a stratification of $V_{2n-1}$ with property $\PAplus$.
Hence, by Lemma \ref{towerofstrat}, there is a good stratification of $V$, contradicting our assumptions.

Hence, there is no scalar $\lambda$ such that $(\overline{u}-\lambda \id)^2=0$.
By Lemma \ref{twosubmoduleslemma}, there are submodules $W_0 \subset W_1$ of $V/V_{2n-1}$ such that
$W_1$ contains the class of $e_n$ modulo $V_{2n-1}$, and both modules $W_0$ and $W_1/W_0$ are monogenous
and have their dimension over $\F$ finite and greater than $1$.
Then, we define $V_{2n}$ and $V_{2n+1}$ as the respective inverse images of $W_0$ and $W_1$ under the canonical projection
of $V$ onto $V/V_{2n-1}$. Then, $V_{2n+1}/V_{2n}$ and $V_{2n}/V_{2n-1}$ are isomorphic to, respectively,
$W_1/W_0$ and $W_0$, and hence both are monogenous and have their dimension over $\F$ finite and greater than $1$.
Finally, $V_{2n+1}$ contains $e_n$.

Hence, by induction we have constructed an increasing sequence $(V_n)_{n \in \N}$ of submodules of $V$ such that
each quotient module $V_n/V_{n-1}$ is monogenous and has its dimension over $\F$ finite and greater than $1$,
and $V_{2n+1}$ contains $e_n$ for all $n \in \N$. The latter property yields $\sum_{n \in \N} V_n=V$,
and we deduce that $(V_n)_{n \in \N}$ is a good stratification of $V^u$. This completes the proof.
\end{proof}

We finish this section with a basic result on the case of a dominant eigenvalue, to be used in the non-torsion case.

\begin{lemma}\label{splitdominant}
Let $u$ be an endomorphism of a vector space $V$, and assume that there is a scalar $\lambda$ such that
$u-\lambda\,\id_V$ has finite rank. Then,
$V^u$ splits into $V^u=W \oplus H$ in which:
\begin{itemize}
\item $W$ is a finite direct sum of monogenous submodules with dimension over $\F$ finite and greater than $1$;
\item There is a scalar $\mu$ such that $\forall x \in H, \; u(x)=\mu\,x$.
\end{itemize}
\end{lemma}

\begin{proof}
If $V$ is finite-dimensional, the result is an obvious consequence of the classification of finitely generated torsion
$\F[t]$-modules. In the rest of the proof, we assume that $V$ is infinite-dimensional.
Set $w:=u-\lambda \id_V$.

Let us choose a finite-dimensional linear subspace $W'$ of $V$ such that $\im w \subset W'$ and $W'+\Ker w=V$.
Let us choose a complementary subspace $H'$ of $W'$ in $V$ such that $H' \subset \Ker w$.
Note that $\forall x \in H', \; u(x)=\lambda x$.
Since $W'$ includes $\im w$ it is a submodule of $V^u$.
By the classification of finitely generated $\F[t]$-modules, there is a scalar $\mu$ together with a splitting
$$W'=E \oplus \underset{i=1}{\overset{n}{\bigoplus}} W_i$$
in which $\forall x \in E, \; u(x)=\mu \,x$ and each submodule $W_i$ is monogenous with (finite) dimension greater than $1$.

Then, there are two cases to consider.
\begin{itemize}
\item If $\mu=\lambda$ then we take $H:=E \oplus H'$ and $W:=\underset{i=1}{\overset{n}{\bigoplus}} W_i$.

\item Assume that $\mu \neq \lambda$. Then, we choose a basis $(e_1,\dots,e_m)$ of the $\F$-vector space $E$ and
then a linearly independent $m$-tuple $(f_1,\dots,f_m)$ of vectors of $H'$, and we re-split $H'=\Vect(f_1,\dots,f_m) \oplus H$
for some linear subspace $H$. Then, we set $W:=\underset{i=1}{\overset{n}{\bigoplus}} W_i \oplus \underset{i=1}{\overset{m}{\bigoplus}}
\Vect(e_i,f_i)$ and we note that $\Vect(e_i,f_i)=\F[t](e_i+f_i)$ is monogenous with dimension $2$ for all $i \in \lcro 1,m\rcro$.
\end{itemize}
\end{proof}

\subsection{The case of non-torsion $\F[t]$-modules}\label{nontorsioncase}

\begin{Def}
Let $V$ be an $\F[t]$-module and $F$ be a free submodule of $V$.
We say that $F$ is \textbf{quasi-maximal} if $V/F$ is a torsion module.
\end{Def}

Equivalently, $F$ is quasi-maximal if and only if there is no non-zero free submodule $F'$ of $V$ such that $F \cap F'=\{0\}$.
Beware that a quasi-maximal free submodule may not be maximal among the free submodules:
for example, in the $\F[t]$-module $\F[t]$, the free submodule $t\F[t]$ is quasi-maximal but it is not a maximal free submodule.

To construct a quasi-maximal free submodule of $V$, it suffices to take
a maximal $\F[t]$-independent subset $A$ of $V$ (which exists thanks to Zorn's lemma)
and to consider the free module $\underset{x \in A}{\sum} \F[t] x$.

In order to prove Theorem \ref{alphaelementarilyTheo} in the remaining case
when $V^u$ is not a torsion module and the dimension of $V$ is countable, the main step consists in the following decomposition of a non-torsion $\F[t]$-module devoid of a good stratification.

\begin{lemma}\label{nontorsiondecomp}
Let $V$ be an $\F[t]$-module. Assume that $V$ is not a torsion module, that $V$ has countable dimension as a vector space over $\F$,
and that $V$ has no good stratification.
Then, there exist submodules $F \subset W$ of $V$ together with a scalar $\lambda$ such that:
\begin{enumerate}[(a)]
\item $F$ is a non-zero free submodule of $V$;
\item $W/F$ has finite dimension over $\F$ and, if nonzero, has a stratification that satisfies condition $\PAplus$;
\item There is a splitting $V=W \oplus H$ such that $\forall x \in H, \; t\,x=\lambda\,x$.
\end{enumerate}
\end{lemma}

\begin{proof}
We start by choosing a quasi-maximal free submodule $F'$ of $V$. Note that $F' \neq \{0\}$ since $V$ is not a torsion module.
The quotient module $V/F'$ is a torsion module whose dimension is at most countable.
We can choose a basis $(e_i)_{i \in I}$ of the free module $F'$ indexed by a subset $I$ of $\N$.
Then, by setting $V_i:=\underset{j \in I, \; j \leq i}{\sum}  \F[t]e_j$, we see that
$(V_i)_{i \in I}$ is a good stratification of $F'$. Hence, if $V/F'$ had a good stratification,
Lemma \ref{towerofstrat} would yield that $V$ has a good stratification, which has been ruled out.
If $F'=V$ then $V$ has a good stratification.

Hence, $V/F'$ is nonzero and it has no good stratification. It follows from Proposition \ref{torsiongoodstrat} that
$V/F'$ is finite-dimensional as a vector space over $\F$
or the endomorphism $x \mapsto t\,x$ of $V/F'$ has a dominant eigenvalue.
In any case, Lemma \ref{splitdominant} yields a scalar $\lambda$ and a module splitting
$$V/F'=K \oplus G$$
in which:
\begin{itemize}
\item Each vector of $K$ is annihilated by $t-\lambda$;
\item The module $G$ splits into a finite direct sum of monogenous submodules, each of which with dimension over $\F$ finite and greater than $1$.
\end{itemize}
In particular, $G$ has a finite stratification that satisfies condition $\PAplus$.

Since $p(t) \mapsto p(t+\lambda)$ is an automorphism of the algebra $\F[t]$, we lose
no generality in assuming that $\lambda=0$, and we shall assume that this condition holds throughout the
remainder of the proof.

Next, we define $V_1$ as the inverse image of $K$ under the canonical projection of $V$ onto $V/F'$.
It follows that $V/V_1$ is isomorphic to $G$, and hence it has a stratification that satisfies $\PAplus$.

Since $F'$ is a free $\F[t]$-module, we can choose an $\F$-linear subspace $F'_0$ of it such that
$F'=\underset{n \in \N}{\bigoplus} t^n F'_0$  and $x \in F'_0 \mapsto t^n x$ is injective for all $n \in \N$.

Next, we consider the inverse image $L$ of $F'_0$ under $x \in V_1 \mapsto t\,x$.
We have $L+F'=V_1$: indeed, given $x \in V_1$, we have $t\,x \in F'$ and hence
$t\,x=x_0+t\,x_1$ for some $x_0 \in F'_0$ and some $x_1 \in F'$, whence $x-x_1 \in L$.

It follows that we can find a linear subspace $H' \subset L$ such that
$$V_1=F' \oplus H',$$
leading to $\forall x \in H', \; t\,x \in F'_0$.
Next, we split $H'$ as follows: we consider the linear mapping $h : x \in H' \mapsto t\,x \in F'_0$,
we denote by $H$ its kernel, and we consider a complementary subspace $F_1$ of $H$ in $H'$.
It follows that $\forall x \in H, \; t\,x=0$, whereas $x \mapsto t\,x$ maps
$F_1$ bijectively onto a linear subspace $F'_1$ of $F'_0$.
Finally, we consider a complementary subspace $F_0$ of $F'_1$ in $F'_0$, and we set
$$F:=\Bigl(\underset{n \in \N}{\bigoplus}\, t^n F_0\Bigr) \oplus \Bigl(\underset{n \in \N}{\bigoplus}\, t^n F_1\Bigr)
=\F[t] F_0 \oplus \F[t] F_{1}=\F[t]F_0 \oplus \F[t]F'_1 \oplus F_1=F' \oplus F_1.$$
Then, $F$ is a free submodule of $V_1$ and
$$V_1=F' \oplus H'=F' \oplus F_1 \oplus H=F \oplus H.$$
Moreover, $F' \subset F \subset V_1$.

Now, we choose a \emph{linear subspace} $G'$ of $V$ which is mapped bijectively onto $G$ under the canonical projection $V \rightarrow V/F'$. Then, since $t\,x=0$ for all $x \in K$, we know that $t\,x \in F'+G' \subset F+G'$ for all $x \in G'$. Moreover,
the definition of $G'$ yields $G' \cap V_1=\{0\}$, whence $G' \cap F=\{0\}$.

Hence, $W:=F \oplus G'$ is a submodule of $V$ and $W/F$ is isomorphic to $G$,
which equals zero or has a stratification that satisfies condition $\PAplus$.

Since $V/F'=K \oplus G$, we have
$$V=V_1\oplus G'=F \oplus H \oplus G'=W \oplus H,$$
which completes the proof.
\end{proof}

We conclude that, in order to establish Theorem \ref{alphaelementarilyTheo} in the case of
a countable-dimensional space and a non-torsion $\F[t]$-module, it only remains to prove the following result:

\begin{prop}\label{lastnontorsionprop}
Let $u$ be an endomorphism of a vector space $V$ with countable dimension. Let $\lambda \in \F$ and $a \in \F$.
Assume that we have a splitting $V^u=W\oplus H$ and a non-zero free submodule $F$ of $W$ such that:
\begin{enumerate}[(a)]
\item The $\F[t]$-module $W/F$ is finite-dimensional as an $\F$-vector space, and if nonzero it has a stratification that satisfies condition $\PAplus$.
\item $\forall x \in H, \; u(x)=\lambda\,x$.
\end{enumerate}
Then, $u$ is $a$-elementarily decomposable.
\end{prop}

To prove this result, the key is to consider the most simple situation, in which
$W=F$ and $F$ is monogenous:

\begin{lemma}[Sewing lemma]\label{sewinglemma}
Let $V$ be a vector space with countable dimension over $\F$, and
$u$ be an endomorphism of $V$.
Assume that we have a module splitting $V^u=V_1 \oplus V_2$
in which:
\begin{itemize}
\item $V_1$ is free, non-zero and monogenous;
\item $u$ induces a scalar multiple of the identity on $V_2$.
\end{itemize}
Let $x$ be a generator of $V_1$ and $a$ be a scalar. Then, there exists an endomorphism $v$ of $V$ such that
$v^2=av$ and, for $u':=u-v$, one has $V=\Vect\bigl((u')^k(x)\bigr)_{k \in \N}$.
\end{lemma}

In short, we have a very specific perturbation of $u$ so as to turn $V$ into the free monogenous $\F[t]$-module generated by $x$.

\begin{proof}
For $n \in \N$, set $e_n:=u^n(x)$.

Assume first that $V_2$ has countable dimension, and
choose a basis $(f_n)_{n \in \N}$ of $V_2$.

If $a=0$, we define $v$ on the basis $(e_n)\coprod (f_n)$ as follows:
for all $n \in \N$, we set
$$\begin{cases}
v(e_{3n})=e_{3n+1}-f_n \\
v(e_{3n+1})=-e_{3n+1}+f_n \\
v(e_{3n+2})=0 \\
v(f_n)=-e_{3n+1}+f_n.
\end{cases}$$
Otherwise, we define it on the same basis as follows:
for all $n \in \N$, we set
$$\begin{cases}
v(e_{3n})=e_{3n+1}-f_n \\
v(e_{3n+1})=0 \\
v(e_{3n+2})=0 \\
v(f_n)=-a e_{3n+1}+a f_n.
\end{cases}$$
In any case, one easily checks that $v^2=a\,v$.
Moreover, in any case, one also checks by induction on $n$ that
$\Vect((u-v)^k(x))_{0 \leq k \leq 4n}=\Vect(e_0,e_1,\dots,e_{3n},f_0,\dots,f_{n-1})$ for all $n \in \N$.
This proves the claimed statement.

Assume finally that $V_2$ has finite dimension $p$. If $p=0$, we simply take $v=0$.
Now, assuming otherwise we choose a basis $(f_0,\dots,f_{p-1})$ of $V_2$.
Then, we slightly modify the above definition of $v$:
\begin{itemize}
\item If $a=0$, we define, for all every non-negative integer $n \leq p-1$,
$$\begin{cases}
v(e_{3n})=e_{3n+1}-f_n \\
v(e_{3n+1})=-e_{3n+1}+f_n \\
v(e_{3n+2})=0 \\
v(f_n)=-e_{3n+1}+f_n
\end{cases}$$
and for every integer $k \geq 3p$ we set $v(e_k)=0$;
\item If $a \neq 0$, we define, for all every non-negative integer $n \leq p-1$,
$$\begin{cases}
v(e_{3n})=e_{3n+1}-f_n \\
v(e_{3n+1})=0 \\
v(e_{3n+2})=0 \\
v(f_n)=-a e_{3n+1}+a f_n
\end{cases}$$
and for every integer $k \geq 3p$ we set $v(e_k)=0$.
\end{itemize}
In any case, it is once more easy to check that $v^2=av$.
Moreover, one proves
by finite induction that $\Vect\bigl((u-v)^k(x)\bigr)_{0 \leq k \leq 4n}=\Vect(e_0,e_1,\dots,e_{3n},f_0,\dots,f_{n-1})$ for all $n \in \lcro 0,p\rcro$,
and then $\Vect\bigl((u-v)^k(x)\bigr)_{0 \leq k \leq q}=\Vect(e_0,e_1,\dots,e_{q-p},f_0,\dots,f_{p-1})$
for all $q \geq 4p$.
Again, the claimed statement is proved in that case.
\end{proof}

\begin{proof}[Proof of Proposition \ref{lastnontorsionprop}]
We split $F=\F[t]x \oplus F'$ for some non-zero vector $x$ and some free submodule $F'$.

We assume first that $F \subsetneq W$.
Let us consider a stratification $(W_1,\dots,W_N)$ of $W/F$ with property $\PAplus$,
with associated dimension sequence $(n_1,\dots,n_N)$
and an associated vector sequence $(\overline{x_1},\dots,\overline{x_N})$
in which $x_i$ denotes a vector of $W$ and $\overline{x_i}$ denotes its class modulo $F$.
We denote by $M$ the dimension of the $\F$-vector space $W/F$ and we set
$$\bfB:=\bigl(x_1,\dots,u^{n_1-1}(x_1),x_2,\dots,u^{n_2-1}(x_2),\dots, u^{n_N-2}(x_N)\bigr)
\quad \text{and} \quad G_1:=\Vect(\bfB).$$
For all $i \in \lcro 1,N\rcro$, there is a polynomial $p_i(t) \in \F[t]$ such that
$u^{n_i}(x_i)$ equals $p_i(t)\,x$ modulo $F'+\Vect\bigl(x_1,\dots,u^{n_1-1}(x_1),x_2,\dots,u^{n_2-1}(x_2),\dots, u^{n_i-1}(x_i)\bigr)$.
We set
$$m:=\max\bigl(0,\deg(p_1(t)),\dots,\deg(p_N(t))\bigr) \quad \text{and} \quad d:=M+m.$$
Then, we consider the linear map
$$f : G_1 \oplus \F_{d-1}[t]x \oplus F' \rightarrow V$$
that sends
$u^{n_k-1}(x_k)$ to $au^{n_k-1}(x_k)-x_{k+1}$ for all $k \in \lcro 1,N-1\rcro$,
that sends all the other vectors of $\bfB$ to $0$, and that sends all the vectors of
$\F_{d-1}[t]x \oplus F'$ to $0$.

Then, we define inductively $(y_1,\dots,y_M)$ by $y_1:=x_1$ and, for all $k \in \lcro 1,M-1\rcro$,
$y_{k+1}:=(u-f)(y_k)$: this makes sense because one proves by induction that, for each $k \in \lcro 1,M-1\rcro$,
the vector  $y_k$ equals the $k$-th vector of $\bfB$ modulo the sum of
$\F_{m+k-2}[t]x\oplus F'$ with the span of the first $k-1$ vectors of $\bfB$.
It follows that $y_M$ equal $u^{n_N-1}(x_N)$ modulo $\F_{d-2}[t]x\oplus F'\oplus G_1$.
Setting
$$G_2:=G_1 \oplus \F y_M \quad \text{and} \quad G_3:=\Vect(y_1,\dots,y_M),$$
we note that
$$W=F \oplus G_2=F \oplus G_3$$
and that $u(y_M)=z+z'$ for some $(z,z')\in (\F_{d-1}[t]x \oplus F') \times G_3$. \\
We finally extend $f$ into a linear map on $G_2 \oplus \F_{d-1}[t]x \oplus F'$ by setting
$$f(y_M):=a y_M+z-x,$$
so that
$(u-f)(y_M)=x$ modulo $G_3$.
Now, set
$$V':=\F[t]t^{d} x\oplus H.$$
Applying the sewing lemma to the endomorphism of $V'$ induced by $u$ and to the vector $t^{d} x$,
we recover an endomorphism $g$ of $V'$ such that $g^2=a g$ and the sequence $\bigl((u-g)^k(t^{d} x)\bigr)_{k \in \N}$ spans
$V'$.

Finally, noting that $V=\F_{d-1}[t] x \oplus V' \oplus G_2 \oplus F'$, we consider the unique endomorphism $v$ of $V$
whose restriction to $V'$ is $g$ and whose restriction to $\F_{d-1}[t]x \oplus F' \oplus G_2$ is $f$.

We claim that $v^2=av$ and that $u-v$ is elementary.
Let us check first the equality on each subspace $V'$, $\F_{d-1}[t]x  \oplus F'$ and $G_2$.
First of all, both $v^2$ and $av$ vanish everywhere on $\F_{d-1}[t]x\oplus F'$.
Next, by the very definition of $g$, we know that $v^2$ and $av$ coincide on $V'$.
Finally, it is easily checked that $v^2(y)=av(y)$ for every vector $y$ in $\bfB$
by using the fact that $n_2,\dots,n_N$ are all greater than $1$; on the other hand, since $z-x$ belongs to $\F_{d-1}[t]x \oplus F'$, we
have $f(z-x)=0$ and it follows that $v\bigl((v-a\id)(y_M)\bigr)=0$, i.e.\ $v^2(y_M)=a v(y_M)$.

Obviously, the module $(F')^{u-v}=(F')^u$ is free.
In order to conclude, we shall simply check that
$\F[t]x\oplus G_3 \oplus H$ is stabilized by $u-v$ and that the module
$(\F[t]x\oplus G_3 \oplus H)^{u-v}$ is free with generator $y_1$.
First of all, we have $(u-v)^i(y_1)=y_{i+1}$ for all $i \in \lcro 1,M-1\rcro$,
and then $(u-v)^M(y_1)=x$ modulo $G_3$.
Then, as $f$ vanishes everywhere on $\F_{d-1}[t]x$, the definitions of $d$ and $v$ show, by induction, that
$(u-v)^k(y_1)=u^{k-M}(x)$ modulo $\Vect\bigl((u-v)^i(y_1)\bigr)_{0 \leq i<k}$ for all $k \in \lcro M,M+d-1\rcro$.
Moreover, the choice of $g$ shows that
$$V' =\Vect\bigl((u-v)^l (u^{d}(x))\bigr)_{l \in \N} \subset \Vect\bigl((u-v)^l(y_1)\bigr)_{l \in \N}
\subset \F[t]x\oplus H \oplus G_{3.}$$
It follows that $\bigl((u-v)^i(y_1)\bigr)_{i \in \N}$ generates the infinite-dimensional vector space
$\F[t]x\oplus H \oplus G_3$, which yields the claimed result.
Therefore, $V=F'\oplus (\F[t]x \oplus H \oplus G_3)$ is a free $\F[t]$-module for the structure induced by $u-v$,
or in other words $u-v$ is elementary.

Finally, assume that $W=F$. Then, we simply split $V=(\F[t]x\oplus H) \oplus F'$, and we apply the sewing lemma to
the endomorphism of $\F[t]x\oplus H$ induced by $u$, which yields an endomorphism $w$ of $\F[t]x\oplus H$
such that $w^2=aw$ and the module $(\F[t]x\oplus H)^{u-w}$ is free. We extend $w$ into an endomorphism $v$ of $V$
that maps every vector of $F'$ to $0$, and we obtain that $v^2=av$ and that $V^{u-v}$ is free.
\end{proof}

Now, the proof of Theorem \ref{alphaelementarilyTheo} is complete over all vector spaces of countable dimension.
Hence, Theorem \ref{theo3} is finally established in all situations.

\section{Special decompositions}\label{specialcasessection}

In this last section, we complete the proofs of Theorems \ref{3squarezero}, \ref{3idemcar2theo} and \ref{LC3}.

\subsection{Sums of three square-zero endomorphisms}\label{3squarezerosection}

Let $u$ be an endomorphism of an infinite-dimensional vector space $V$.
Let us first apply Theorems \ref{theo3}, \ref{dominanteigenvalueCN} and \ref{dominanteigenvalueCS}
to $p_1=p_2=p_3=t^2$. Here, $\tr p_1+\tr p_2+\tr p_3=0$, and
a scalar is a $(p_1,p_2,p_3)$-sum if and only if it equals $0$.
Hence:
\begin{itemize}
\item If $u$ has no dominant eigenvalue then it is the sum of three square-zero endomorphisms, by Theorem \ref{theo3}.
\item If $u$ has a dominant eigenvalue $\lambda$, it is the sum of three square-zero endomorphisms only if
$\lambda=0$ or $\F$ has characteristic $2$, according to Theorem \ref{dominanteigenvalueCN}.
\item If $u$ has a dominant eigenvalue $\lambda$ such that $u-\lambda \id_V$ has infinite rank, and either $\lambda=0$ or $\F$
has characteristic $2$, then Theorem \ref{dominanteigenvalueCS} yields that $u$ is the sum of three square-zero endomorphisms.
\end{itemize}
It only remains to tackle the case when
$u$ splits as $u=\lambda\id_V+w$ for some finite-rank endomorphism $w$ of $V$, and either $\F$ has characteristic $2$ or
$\lambda=0$.

Assume that $u$ is the sum of three square-zero endomorphisms.
Let $A$ be a square matrix in the class $[w]$ (see Section \ref{finiteranksection}), with size $n \times n$.
By Theorem \ref{theofiniterank}, there is a non-negative integer $q$ such that $(A+\lambda I_n) \oplus \lambda I_q$
is the sum of three square-zero matrices. Hence, its trace equals zero, leading to $\tr w+(n+q)\lambda=0$.
If $\F$ has characteristic $2$, this yields $\tr w \in \{0,\lambda\}$, otherwise we know that $\lambda=0$ and hence $\tr w=0$.

Conversely, by Corollaries 1.5 and 1.6 of \cite{dSP3squarezero} we have the following results:
\begin{itemize}
\item Every finite-rank endomorphism $w$ of $V$ with trace zero is the sum of three square-zero endomorphisms of $V$.

\item If $\F$ has characteristic $2$, then, for all $\lambda \in \F$ and every finite-rank endomorphism $w$ of $V$ with trace in $\{0,\lambda\}$,
the endomorphism $\lambda \id_V+w$ is the sum of three square-zero endomorphisms of $V$.
\end{itemize}

This completes the proof of Theorem \ref{3squarezero}.

\subsection{Sums of three idempotents over a field with characteristic $2$}

Here, we assume that the underlying field $\F$ has characteristic $2$.
We put $p_1=p_2=p_3=t^2-t$. Since $\F$ has characteristic $2$ the equation $2\lambda=\tr p_1+\tr p_2+\tr p_3$
has no solution in $\F$. Moreover, a scalar is a $(p_1,p_2,p_3)$-sum if and only if it belongs to $\{0_\F,1_\F\}$.
Hence:
\begin{itemize}
\item If $u$ has no dominant eigenvalue then it is the sum of three idempotent endomorphisms, by Theorem \ref{theo3}.
\item If $u$ has a dominant eigenvalue $\lambda$, then it is the sum of three idempotent endomorphisms only if
$\lambda \in \{0_\F,1_\F\}$, according to Theorem \ref{dominanteigenvalueCN}.
\item If $u$ has a dominant eigenvalue $\lambda\in \{0_\F,1_\F\}$ such that $u-\lambda \id_V$ has infinite rank, then $u$ is the sum of three idempotent endomorphisms, by Theorem \ref{dominanteigenvalueCS}.
\end{itemize}

It remains to consider the case when $u=\lambda \id_V+w$ for some finite-rank endomorphism $w$
of $V$ and some $\lambda \in \{0_\F,1_\F\}$. Yet, every idempotent square matrix with entries in $\F$ has its trace in $\{0_\F,1_\F\}$.
With the same line of reasoning as in Section \ref{3squarezerosection}, we obtain that if $u$ is the sum of three idempotent endomorphisms
of $V$ then $m\lambda+\tr(w) \in \{0_\F,1_\F\}$ for some non-negative integer $m$, whence $\tr(w) \in \{0_\F,1_\F\}$.

Conversely, by Corollary 1.7 from \cite{dSP3squarezero}, for every $\lambda \in \{0_\F,1_\F\}$
and every finite-rank endomorphism $w$ of $V$ such that $\tr w \in \{0_\F,1_\F\}$, the endomorphism
$\lambda \id_V+w$ is the sum of three idempotent endomorphisms.

This completes the proof of Theorem \ref{3idemcar2theo}.

\subsection{Every endomorphism is a linear combination of three idempotents}

Theorem \ref{LC3} is already known in the finite-dimensional case: See \cite{dSPLC3}.
Now, we complete the infinite-dimensional case.

Let $u$ be an endomorphism of an infinite-dimensional vector space $V$.
If $u$ has no dominant eigenvalue, then $u$ is the sum of three idempotent endomorphisms, by Theorem \ref{theo3}.
Assume now that $u$ has a dominant eigenvalue $\lambda$ such that $u-\lambda\,\id_V$ has infinite rank.
We can split $\lambda=a_1+a_2+a_3$ where each $a_i$ is a scalar in $\F$.
For all $i \in \lcro 1,3\rcro$, we set $p_i:=t^2-a_it$ if $a_i \neq 0$, otherwise we set $p_i:=t^2-t$.
Hence, $\lambda$ is a $(p_1,p_2,p_3)$-sum, and by Theorem \ref{dominanteigenvalueCS} we conclude that
$u$ is a $(p_1,p_2,p_3)$-sum, which yields that it is a linear combination of three idempotents.

Assume finally that $u$ has a dominant eigenvalue $\lambda$ for which $u-\lambda \,\id_V$ has finite rank.
Then, Corollary 6.2 from \cite{dSP3squarezero} yields that $u$ is a linear combination of three idempotents.

This completes the proof of Theorem \ref{LC3}.

\end{document}